\documentclass[11pt]{amsart}

\author[C.~Sanna]{Carlo Sanna}
\thanks{$\dagger\,$C.~Sanna is supported by a postdoctoral fellowship of INdAM and is a member of the INdAM group GNSAGA}
\address{Universit\`a di Genova\\Department of Mathematics\\Genova, Italy}
\email{carlo.sanna.dev@gmail.com}

\title{A note on product sets of random sets}

\usepackage{amsmath}
\usepackage{amssymb}
\usepackage{amsthm}
\usepackage{geometry}
\usepackage{color}
\usepackage{hyperref}
\usepackage{enumerate}
\hypersetup{colorlinks=true}

\newtheorem{thm}{Theorem}[section]

\newtheorem{lem}[thm]{Lemma}
\theoremstyle{remark}

\begin{document}

\begin{abstract}
Given two sets of positive integers $A$ and $B$, let $AB := \{ab : a \in A,\, b \in B\}$ be their \emph{product set} and put $A^k := A \cdots A$ ($k$ times $A$) for any positive integer $k$.
Moreover, for every positive integer $n$ and every $\alpha \in [0,1]$, let $\mathcal{B}(n, \alpha)$ denote the probabilistic model in which a random set $A \subseteq \{1, \dots, n\}$ is constructed by choosing independently every element of $\{1, \dots, n\}$ with probability $\alpha$.
We prove that if $A_1, \dots, A_s$ are random sets in $\mathcal{B}(n_1, \alpha_1), \dots, \mathcal{B}(n_s, \alpha_s)$, respectively, $k_1, \dots, k_s$ are fixed positive integers, $\alpha_i n_i \to +\infty$, and $1/\alpha_i$ does not grow too fast in terms of a product of $\log n_j$; then $|A_1^{k_1} \cdots A_s^{k_s}| \sim \frac{|A_1|^{k_1}}{k_1!}\cdots\frac{|A_s|^{k_s}}{k_s!}$ with probability $1 - o(1)$.
This is a generalization of a result of Cilleruelo, Ramana, and Ramar\'e, who considered the case $s = 1$ and $k_1 = 2$.
\end{abstract}

\maketitle

\section{Introduction}

Given two sets of positive integers $A$ and $B$, let $AB := \{ab : a \in A,\, b \in B\}$ be their \emph{product set} and put $A^k := A \cdots A$ ($k$ times $A$) for any positive integer $k$.

Problems involving the cardinalities of product sets have been considered by many researchers.
For example, the study of $M_n := |\{1,\dots,n\}^2|$ as $n \to +\infty$ is known as the ``multiplicative table problem'' and was started by Erd\H{o}s~\cite{MR73619, MR0126424}.
The exact order of magnitude of $M_n$ was determined by Ford~\cite{MR2434882} following earlier work of Tenenbaum~\cite{MR928635}.
Furthermore, Koukoulopoulos~\cite{MR3187928} provided uniform bounds for $|\{1, \dots, n_1\}\cdots\{1, \dots, n_s\}|$ holding for a wide range of $n_1, \dots, n_s$.
Cilleruelo, Ramana, and Ramar\'e~\cite{MR3640773} proved asymptotics or bounds for \mbox{$|(A \cap \{1,\dots,n\})^2|$} when $A$ is the set of shifted prime numbers, the set of sums of two squares, or the set of shifted sums of two squares.

For every positive integer $n$ and every $\alpha \in [0,1]$, let $\mathcal{B}(n, \alpha)$ denote the probabilistic model in which a random set $A \subseteq \{1, \dots, n\}$ is constructed by choosing independently every element of $\{1, \dots, n\}$ with probability $\alpha$.
Cilleruelo, Ramana, and Ramar\'e~\cite{MR3640773} proved the following:

\begin{thm}\label{thm:CRR}
Let $A$ be a random set in $\mathcal{B}(n, \alpha)$.
If $\alpha n \to +\infty$ and $\alpha = o((\log n)^{-1/2})$, then $|A^2| \sim \frac{|A|^2}{2}$ with probability $1 - o(1)$.
\end{thm}

The contribution of this paper is the following generalization of Theorem~\ref{thm:CRR}.

\begin{thm}\label{thm:main}
Let $A_1, \dots, A_s$ be random sets in $\mathcal{B}(n_1, \alpha_1), \dots, \mathcal{B}(n_s, \alpha_s)$, respectively; and let $k_1, \dots, k_s$ be fixed positive integers.
If $\alpha_i n_i \to +\infty$ and
\begin{equation*}
\alpha_i = o\!\left(\left((\log n_1)^{k_1 - 1}\prod_{i \,=\, 2}^s (\log n_i)^{k_i}\right)^{-(k_1 + \cdots + k_s - 1)/2}\right) ,
\end{equation*}
for $i=1,\dots,s$, then $|A_1^{k_1} \cdots A_s^{k_s}| \sim \frac{|A_1|^{k_1}}{k_1!}\cdots\frac{|A_s|^{k_s}}{k_s!}$ with probability $1 - o(1)$.
\end{thm}

\section{Notation}

We employ the Landau--Bachmann ``Big Oh'' and ``little oh'' notations $O$ and $o$, as well as the associated Vinogradov symbol $\ll$, with their usual meanings.
Any dependence of implied constants is explicitly stated or indicated with subscripts.
For real random variables $X$ and $Y$, we say that ``$X = o(Y)$ with probability $1 - o(1)$'' if $\mathbb{P}(|X| \geq \varepsilon|Y|) = o_\varepsilon(1)$ for every $\varepsilon > 0$, and that ``$X \sim Y$ with probability $1 - o(1)$'' if $X = Y + o(Y)$ with probability $1 - o(1)$.

\section{Preliminaries}

In this section we collect some preliminary results not directly related with product sets.

\begin{lem}\label{lem:rankin}
Let $m$ be a positive integer.
We have
\begin{equation*}
\sum_{a_1 \cdots a_m \,\leq\, x} \frac1{a_1 \cdots a_m} \ll_m (\log x)^m ,
\end{equation*}
for all $x \geq 2$.
\end{lem}
\begin{proof}
This is a standard application of Rankin's method: For $t := m / \log x$, we have
\begin{align*}
\sum_{a_1 \cdots a_m \,\leq\, x} \frac1{a_1 \cdots a_m} &\leq x^t \sum_{a_1 \cdots a_m \,\leq\, x} \frac1{(a_1 \cdots a_m)^{1 + t}} \leq x^t \left(\sum_{a \,=\, 1}^\infty \frac1{a^{1 + t}}\right)^m \\
&\leq x^t \left(1 + \frac1{t}\right)^m \ll_m (\log x)^m ,
\end{align*}
as claimed.
\end{proof}

The next lemma is an upper bound on the number of matrices of positive integers with bounded products of rows and columns.

\begin{lem}\label{lem:matrix}
Let $m$ and $n$ be positive integers.
Then, for all $x_1, \dots, x_n, y_1, \dots, y_m \geq 2$, the number of $m \times n$ matrices $(c_{i,j})$ of positive integers satisfying $\prod_{i=1}^m c_{i,h} \leq x_h$ and $\prod_{j=1}^n c_{k,j} \leq y_k$, for $h=1,\dots,n$ and $k=1,\dots,m$, is at most
\begin{equation}\label{equ:matrixbound}
O_{m,n}\!\left(\left(\prod_{i \,=\, 1}^n x_i \prod_{j \,=\, 1}^{m} y_j \right)^{1/2} \left(\prod_{i \,=\, 1}^{n-1} \log x_i\right)^{m-1}\right)
\end{equation}
\end{lem}
\begin{proof}
We follow the same arguments of~\cite[p.~380]{MR1815216}, where the case $m = n$ and $x_1 = \cdots = x_n = y_1 = \cdots = y_m$ is proved.

The number of choices for $c_{m,n}$ is at most
\begin{equation*}
\min\!\left(\frac{x_n}{\prod_{i=1}^{m-1} c_{i,n}}, \frac{y_m}{\prod_{j=1}^{n-1} c_{m,j}}\right) \leq \left(\frac{x_n y_m}{\prod_{i=1}^{m-1} c_{i,n}\prod_{j=1}^{n-1} c_{m,j}}\right)^{1/2} .
\end{equation*}
We shall sum this latter quantity over all the choices of $c_{i,n}$ and $c_{m,j}$, with $i=1,\dots,m-1$ and $j=1,\dots,n-1$.
Since $c_{i,n} \leq y_i / \prod_{k=1}^{n-1} c_{i,k}$ and $c_{m,j} \leq x_j / \prod_{h=1}^{m - 1} c_{h,j}$, we have
\begin{equation*}
\sum_{c_{i,n}} \frac1{c_{i,n}^{1/2}} \ll \left(\frac{y_i}{\prod_{k=1}^{n-1}c_{i,k}}\right)^{1/2} \quad\text{and}\;\;\quad \sum_{c_{m,j}} \frac1{c_{m,j}^{1/2}} \ll \left(\frac{x_j}{\prod_{h=1}^{m-1} c_{h,j}}\right)^{1/2} ,
\end{equation*}
for $i=1,\dots,m-1$ and $j=1,\dots,n-1$.
Consequently,
\begin{align*}
\sum_{\substack{c_{1,n}, \dots, c_{m-1,n} \\ c_{m,1}, \dots, c_{m,n-1}}} &\left(\frac{x_n y_m}{\prod_{i=1}^{m-1} c_{i,n}\prod_{j=1}^{n-1} c_{m,j}}\right)^{1/2} \leq (x_n y_m)^{1/2} \prod_{i \,=\, 1}^{m - 1}\left(\sum_{c_{i,n}} \frac1{c_{i,n}^{1/2}}\right) \prod_{j \,=\, 1}^{n - 1} \left(\sum_{c_{m,j}} \frac1{c_{m,j}^{1/2}}\right) \\
&\ll_{m,n} \left(\prod_{j\,=\,1}^n x_j \prod_{i\,=\,1}^m y_i \right)^{1/2} \left(\prod_{h \,=\, 1}^{m - 1} \prod_{k \,=\, 1}^{n-1}c_{h,k}\right)^{-1} .
\end{align*}
It remains only to sum over all the possibilities for $c_{h,k}$, with $h=1,\dots,m-1$ and $k=1,\dots,n-1$.
Thanks to Lemma~\ref{lem:rankin}, we have
\begin{equation*}
\sum_{c_{h,k}} \left(\,\prod_{h \,=\, 1}^{m - 1} \prod_{k \,=\, 1}^{n-1}c_{h,k}\right)^{-1} \leq \prod_{k \,=\, 1}^{n - 1} \sum_{c_{1,k} \cdots c_{m-1,k} \,\leq\, x_k} \frac1{c_{1,k} \cdots c_{m-1,k}} \ll_{m,n} \left(\,\prod_{k \,=\, 1}^{n - 1} \log x_k\right)^{m-1} ,
\end{equation*}
and the desired result follows.
\end{proof}

The next lemma is an upper bound for the number of solutions of a certain multiplicative equation with bounded factors.

\begin{lem}\label{lem:energy}
Let $m$ and $n$ be positive integers.
Then, for all $x_1, \dots, x_n, y_1, \dots, y_m \geq 2$, the number of solutions of the equation $a_1 \cdots a_n = b_1 \cdots b_m$, where $a_1, \dots, a_n, b_1, \dots, b_m$ are positive integers satisfying $a_i \leq x_i$ and $b_j \leq y_j$, for $i=1,\dots,n$ and $j=1,\dots,m$, is at most~\eqref{equ:matrixbound}.
\end{lem}
\begin{proof}
If $a_1 \cdots a_n = b_1 \cdots b_m$ then there exists a $m \times n$ matrix of positive integers $(c_{i,j})$ such that $a_h = \prod_{i=1}^m c_{i,h}$ and $b_k = \prod_{j=1}^n c_{k,j}$, for $h=1,\dots,n$ and $k=1,\dots,m$.
Indeed, $a_1 \mid \prod_{i=1}^m b_i$ implies the existence of positive integers $c_{1,1}, \dots, c_{m,1}$ such that $a_1 = \prod_{i=1}^m c_{i,1}$ and $c_{i,1} \mid b_i$, for $i = 1,\dots,m$.
Then $a_2 \mid \prod_{i=1}^m b_i / c_{i,1}$, which similary implies the existence of positive integers $c_{1,2}, \dots, c_{m,2}$ such that $a_2 = \prod_{i=1}^m c_{i,2}$ and $c_{i,1} c_{i,2} \mid b_i$, for $i = 1,\dots,m$.
Then $a_3 \mid \prod_{i=1}^m b_i / (c_{i,1}c_{i,2})$, and so on, until $a_n = \prod_{i=1}^m b_i / (\prod_{j=1}^{n-1} c_{i,j})$, when we set $c_{i,n} := b_i / \prod_{j=1}^{n-1} c_{i,j}$ for $i=1,\dots,m$.
Applying Lemma~\ref{lem:matrix} we get the desired result.
\end{proof}

\section{Proof of Theorem~\ref{thm:main}}

First, we need an asymptotic for the $k$th power of the size of a random set $A$ in $\mathcal{B}(n, \alpha)$.

\begin{lem}\label{lem:Ak}
Let $A$ be a random set in $\mathcal{B}(n, \alpha)$, and fix an integer $k \geq 1$.
If $\alpha n \to +\infty$, then:
\begin{enumerate}[(i)]
\vspace{0.2em}
\setlength\itemsep{0.5em}
\item $\mathbb{E}(|A|^k) \sim (\alpha n)^k$; and
\item $|A|^k \sim (\alpha n)^k$ with probability $1 - o_k(1)$.
\end{enumerate}
\end{lem}
\begin{proof}
Clearly, $|A|$ follows a binomial distribution with $n$ trials and probability of success $\alpha$.
Consequently, (i) is known (see, e.g.,~\cite[Eq.~(4.1)]{MR2447945}).
In~turn, (i) implies that 
\begin{equation*}
\mathbb{V}(|A|^k) = \mathbb{E}(|A|^{2k}) - \mathbb{E}(|A|^k)^2 = o_k(\mathbb{E}(|A|^k)^2) .
\end{equation*}
Hence, by Chebyshev's inequality, for every $\varepsilon > 0$ we have
\begin{equation*}
\mathbb{P}\!\left(\big||A|^k - \mathbb{E}(|A|^k)\big| \geq \varepsilon\, \mathbb{E}(|A|^k)\right) \leq \frac{\mathbb{V}(|A|^k)}{(\varepsilon\, \mathbb{E}(|A|^k))^2} = o_{k,\varepsilon}(1) ,
\end{equation*}
so that $|A|^k \sim \mathbb{E}(|A|^k) \sim (\alpha n)^k$ with probability $1 - o_k(1)$.
\end{proof}

The next lemma is an easy bound on the size of a product set.

\begin{lem}\label{lem:prodbound}
Let $A_1, \dots, A_s$ be finite sets of positive integers, and let $k_1,\dots,k_s \geq 1$ be integers.
Then
\begin{equation*}
\left|\prod_{i \,=\, 1} A_i^{k_i}\right| \leq \prod_{i \,=\, 1}^s \binom{|A_i| + k_i - 1}{k_i} .
\end{equation*}
\end{lem}
\begin{proof}
The claim follows easily considering that $\binom{|A| + k - 1}{k}$ is the number of unordered $k$-tuples of elements from a set $A$.
\end{proof}

For the rest of this section, let $A_1, \dots, A_s$ be random sets in $\mathcal{B}(n_1, \alpha_1), \dots, \mathcal{B}(n_s, \alpha_s)$, \mbox{respectively}; and let $k_1, \dots, k_s$ be fixed positive integers.
Also, assume $\alpha_i n_i \to +\infty$ and
\begin{equation}\label{equ:BigProd}
\alpha_i = o\!\left(\left((\log n_1)^{k_1 - 1}\prod_{i \,=\, 2}^s (\log n_i)^{k_i}\right)^{-(k_1 + \cdots + k_s - 1)/2}\right) ,
\end{equation}
for $i=1,\dots,s$.
For brevity, we will omit the dependence of implied constants from $k_1, \dots, k_s$.

\begin{lem}\label{lem:Ex}
We have $\mathbb{E}(|A_1^{k_1} \cdots A_s^{k_s}|) \sim \frac{(\alpha_1 n_1)^{k_1}}{k_1!} \cdots \frac{(\alpha_s n_s)^{k_s}}{k_s!}$.
\end{lem}
\begin{proof}
Hereafter, in operator subscripts, let $\boldsymbol{a} := (\boldsymbol{a}_1, \dots, \boldsymbol{a}_s)$, where each $\boldsymbol{a}_i := \{a_{i,1}, \dots, a_{i,k_i}\}$ runs over the unordered $k_i$-tuples of elements of $\{1, \dots, n_i\}$.
Also, put $\|\boldsymbol{a}\| := \prod_{i=1}^s \prod_{j=1}^{k_i} a_{i,j}$.
With this notation, for each positive integer $x$, we have
\begin{equation*}
\mathbb{P}(x \in A_1^{k_1}\cdots A_s^{k_s}) = \mathbb{P}\!\left(\bigvee_{\|\boldsymbol{a}\| \,=\, x} E_{\boldsymbol{a}} \right) ,
\end{equation*}
where
\begin{equation*}
E_{\boldsymbol{a}} := \bigwedge_{i \,=\, 1}^s (\boldsymbol{a}_i \subseteq A_i) .
\end{equation*}
Consequently, by Bonferroni inequalities, we have
\begin{align*}
\mathbb{P}(x \in A_1^{k_1}\cdots A_s^{k_s}) &= \mathbb{P}\!\left(\sideset{\;}{^*}\bigvee_{\|\boldsymbol{a}\| \,=\, x} E_{\boldsymbol{a}} \right) + O\!\left(\sideset{\,\;}{^{**}}\sum_{\|\boldsymbol{a}\| \,=\, x} \mathbb{P}(E_{\boldsymbol{a}})\right) \\
&= \sideset{\;}{^*}\sum_{\|\boldsymbol{a}\| \,=\, x} \mathbb{P}(E_{\boldsymbol{a}}) + O\!\left(\sideset{\;}{^*}\sum_{\substack{\boldsymbol{a} \,\neq\, \boldsymbol{a}^\prime \\ \|\boldsymbol{a}\| \,=\, \|\boldsymbol{a}^\prime\| \,=\, x}} \mathbb{P}(E_{\boldsymbol{a}} \wedge E_{\boldsymbol{a}^\prime})\right) + O\!\left(\sideset{\;\,}{^{**}}\sum_{\|\boldsymbol{a}\| \,=\, x} \mathbb{P}(E_{\boldsymbol{a}})\right) ,
\end{align*}
where the superscript $^*$ denotes the constraint $|\boldsymbol{a}_i| = k_i$ for every $i \in \{1, \dots, s\}$, the superscript $^{**}$ denotes the complementary constrain $|\boldsymbol{a}_i| < k_i$ for at least one $i \in \{1, \dots, k\}$, and $\boldsymbol{a}^\prime := (\boldsymbol{a}_1^\prime, \dots, \boldsymbol{a}_s^\prime)$ follows the same conventions of $\boldsymbol{a}$.
Therefore,
\begin{align}\label{equ:Ex1}
\mathbb{E}(|A_1^{k_1} \cdots A_s^{k_s}|) &= \sum_{x \,\leq\, n_1^{k_1} \,\cdots\, n_s^{k_s}} \mathbb{P}(x \in A_1^{k_1}\cdots A_s^{k_s}) \\
&= \sideset{}{^*}\sum_{\boldsymbol{a}} \mathbb{P}(E_{\boldsymbol{a}}) + O\!\left(\sideset{\;}{^*}\sum_{\substack{\boldsymbol{a} \,\neq\, \boldsymbol{a}^\prime \\ \|\boldsymbol{a}\| \,=\, \|\boldsymbol{a}^\prime\|}} \mathbb{P}(E_{\boldsymbol{a}} \wedge E_{\boldsymbol{a}^\prime})\right) + O\!\left(\sideset{}{^{**}}\sum_{\boldsymbol{a}} \mathbb{P}(E_{\boldsymbol{a}})\right) \nonumber .
\end{align}
Since $A_1, \dots, A_s$ are independent and each $A_i$ belongs to $\mathcal{B}(n_i, \alpha_i)$, we have
\begin{equation*}
\mathbb{P}(E_{\boldsymbol{a}}) = \bigwedge_{i \,=\, 1}^s \mathbb{P}(\boldsymbol{a}_i \subseteq A_i) = \prod_{i \,=\, 1}^s \alpha_i^{|\boldsymbol{a}_i|} .
\end{equation*}
Hence, for every positive integers $m_1, \dots, m_s$, with $m_i \leq k_i$, we have
\begin{equation*}
\sum_{\boldsymbol{a} \,:\, |\boldsymbol{a}_i| \,=\, m_i} \mathbb{P}(E_{\boldsymbol{a}}) = \sum_{\boldsymbol{a} \,:\, |\boldsymbol{a}_i| \,=\, m_i} \prod_{i \,=\, 1}^s \alpha_i^{m_i} = \prod_{i \,=\, 1}^s \alpha_i^{m_i} \sum_{|\boldsymbol{a}_i| \,=\, m_i} 1 = \prod_{i \,=\, 1}^s \alpha_i^{m_i} \binom{n_i}{m_i}\binom{k_i - 1}{m_i - 1} ,
\end{equation*}
where we used the fact that the number of unordered $k$-tuples of elements of $\{1,\dots,n\}$ having cardinality equal to $m$ is $\binom{n}{m}\binom{k - 1}{m - 1}$.
Therefore,
\begin{equation}\label{equ:Ex2}
\sideset{}{^*}\sum_{\boldsymbol{a}} \mathbb{P}(E_{\boldsymbol{a}}) \sim \prod_{i \,=\, 1}^s \frac{(\alpha_i n_i)^{k_i}}{k_i!} \quad\text{and}\quad \sideset{}{^{**}}\sum_{\boldsymbol{a}} \mathbb{P}(E_{\boldsymbol{a}}) = o\!\left(\prod_{i \,=\, 1}^s (\alpha_i n_i)^{k_i}\right) ,
\end{equation}
as $\alpha_i n_i \to +\infty$, for $i=1,\dots,s$.
We have
\begin{equation}\label{equ:P2}
\mathbb{P}(E_{\boldsymbol{a}} \wedge E_{\boldsymbol{a}^\prime}) = \prod_{i \,=\, 1}^s \mathbb{P}(\boldsymbol{a}_i \cup \boldsymbol{a}_i^\prime \subseteq A_i) = \prod_{i \,=\, 1}^s \alpha_i^{|\boldsymbol{a}_i \,\cup\, \boldsymbol{a}_i^\prime|} .
\end{equation}
Suppose that $\boldsymbol{a}$ and $\boldsymbol{a}^\prime$, with $\boldsymbol{a} \neq \boldsymbol{a}^\prime$ and $\|\boldsymbol{a}\| = \|\boldsymbol{a}^\prime\|$, satisfy the condition of $^{*}$, that is, $|\boldsymbol{a}_i| = |\boldsymbol{a}_i^\prime| = k_i$ for $i=1,\dots,s$.
We shall find an upper bound for~\eqref{equ:P2}.
Clearly, $|\boldsymbol{a}_i \cup \boldsymbol{a}_i^\prime| \geq |\boldsymbol{a}_i| \geq k_i$ for $i = 1,\dots,s$.
Moreover, since $\boldsymbol{a} \neq \boldsymbol{a}^\prime$, there exists $i_1 \in \{1, \dots, s\}$ such that $\boldsymbol{a}_{i_1} \neq \boldsymbol{a}_{i_1}^\prime$.
Since $|\boldsymbol{a}_{i_1}| = |\boldsymbol{a}_{i_1}^\prime| = k_i$, it follows that $|\boldsymbol{a}_{i_1} \cup \boldsymbol{a}_{i_1}^\prime| \geq k_{i_1} + 1$.
On the one hand, if there exists $i_2 \in \{1,\dots,s\} \setminus \{i_1\}$ such that $\boldsymbol{a}_{i_2} \neq \boldsymbol{a}_{i_2}^\prime$, then, similarly, we have $|\boldsymbol{a}_{i_2} \cup \boldsymbol{a}_{i_2}^\prime| \geq k_{i_2} + 1$.
Hence,
\begin{equation*}
\mathbb{P}(E_{\boldsymbol{a}} \wedge E_{\boldsymbol{a}^\prime}) \leq \alpha_{i_1} \alpha_{i_2} \prod_{i \,=\, 1}^s \alpha_i^{k_i} .
\end{equation*}
On the other hand, if $\boldsymbol{a}_i = \boldsymbol{a}_i^\prime$ for every $i \in \{1,\dots,s\} \setminus \{i_1\}$, then from $\|\boldsymbol{a}\| = \|\boldsymbol{a}^\prime\|$ it follows that $\prod_{j=1}^{k_{i_1}} a_{i_1, j} = \prod_{j=1}^{k_{i_1}} a_{i_1, j}^\prime$.
In turn, this implies that $|\boldsymbol{a}_{i_1} \cup \boldsymbol{a}_{i_1}^\prime| \geq k_{i_1} + 2$.
Hence,
\begin{equation*}
\mathbb{P}(E_{\boldsymbol{a}} \wedge E_{\boldsymbol{a}^\prime}) \leq \alpha_{i_1}^2 \prod_{i \,=\, 1}^s \alpha_i^{k_i} .
\end{equation*}
Therefore, using Lemma~\ref{lem:energy} and recalling~\eqref{equ:BigProd}, we obtain
\begin{align}\label{equ:Ex3}
\sideset{\;}{^*}\sum_{\substack{\boldsymbol{a} \,\neq\, \boldsymbol{a}^\prime \\ \|\boldsymbol{a}\| \,=\, \|\boldsymbol{a}^\prime\|}} &\mathbb{P}(E_{\boldsymbol{a}} \wedge E_{\boldsymbol{a}^\prime}) \leq \left(\max_{1 \,\leq\, i,j \,\leq\, s} \alpha_i \alpha_j\right) \prod_{i \,=\, 1}^s \alpha_i^{k_i} \sum_{\|\boldsymbol{a}\| \,=\, \|\boldsymbol{a}^\prime\|} 1 \\
&\ll \left(\max_{1 \,\leq\, i,j \,\leq\, s} \alpha_i \alpha_j\right) \left((\log n_1)^{k_1 - 1}\prod_{i \,=\, 2}^s (\log n_i)^{k_i}\right)^{k_1 + \cdots + k_s - 1} \prod_{i \,=\, 1}^s (\alpha_i n_i)^{k_i} \nonumber\\
&= o\!\left(\prod_{i \,=\, 1}^s (\alpha_i n_i)^{k_i}\right) . \nonumber
\end{align}
Finally, putting together~\eqref{equ:Ex1}, \eqref{equ:Ex2}, and~\eqref{equ:Ex3}, we obtain the desired claim.
\end{proof}

\begin{proof}[Proof of Theorem~\ref{thm:main}]
Define the random variable
\begin{equation*}
X := \prod_{i \,=\, 1}^s \binom{|A_i| + k_i - 1}{k_i} - \left|\prod_{i \,=\, 1}^s A_i^{k_i} \right| .
\end{equation*}
Thanks to Lemma~\ref{lem:prodbound}, we know that $X$ is nonnegative.
Moreover, from Lemma~\ref{lem:Ak}(i) and Lemma~\ref{lem:Ex}, it follows that
\begin{equation*}
\mathbb{E}(X) = o\!\left(\prod_{i \,=\, 1}^s (\alpha_i n_i)^{k_i}\right) .
\end{equation*}
Hence, for every $\varepsilon > 0$, by Markov's inequality, we get
\begin{equation*}
\mathbb{P}\!\left(X \geq \varepsilon \prod_{i \,=\, 1}^s (\alpha_i n_i)^{k_i}\right) \leq \frac{\mathbb{E}(X)}{\varepsilon \prod_{i = 1}^s (\alpha_i n_i)^{k_i}} = o_\varepsilon(1) ,
\end{equation*}
which in turn implies $X = o\!\left(\prod_{i \,=\, 1}^s (\alpha_i n_i)^{k_i}\right)$ with probability $1 - o(1)$.
Therefore, by Lemma~\ref{lem:Ak}(ii),
\begin{equation*}
\left|\prod_{i \,=\, 1}^s A_i^{k_i} \right| = \prod_{i \,=\, 1}^s \binom{|A_i| + k_i - 1}{k_i} - X = \prod_{i \,=\, 1}^s \frac{|A_i|^{k_i}}{k_i!} + o\!\left(\prod_{i \,=\, 1}^s |A_i|^{k_i}\right) ,
\end{equation*}
with probability $1 - o(1)$, as claimed.
\end{proof}

\bibliographystyle{amsplain}
\bibliography{temp}

\end{document}